\documentclass[11pt]{amsart}
\usepackage[T1]{fontenc}
\usepackage{lmodern}
\usepackage[margin=1in]{geometry}
\usepackage{amsmath,amssymb,booktabs,array,tikz}
\usepackage[hidelinks]{hyperref}
\allowdisplaybreaks
\newtheorem{theorem}{Theorem}
\newcommand{\norm}[1]{\left\lVert #1\right\rVert_2}
\renewcommand{\d}{\mathop{}\!\mathrm{d}}
\begin{document}
\title{Elementary Proofs of Commonality of the House Graph and \(K_4\) with One Edge Subdivided}
\date{}

\author{Hao Chen}
\thanks{School of Mathematical Sciences, Soochow University,
Suzhou 215006, China. E-mail:
\href{mailto:chenhao@suda.edu.cn}{\texttt{chenhao@suda.edu.cn}}}
\maketitle
\begin{abstract}
A graph \(H\) is called \emph{common} if the number of monochromatic
copies of \(H\) in a \(2\)-edge-coloring of a large complete graph is
asymptotically minimized by the random coloring with an equal proportion
of each color. The house graph and
\(K_4\) with one edge subdivided once are known to be common by the flag algebra method. Grzesik, Lee, Lidick\'y,
and Volec asked for ``human-friendly'' proofs of the commonality of these
two graphs. In this short note, we give flag-algebra-free proofs that
both graphs are common.
\end{abstract}

\section{Introduction}
The Ramsey multiplicity $M(H,n)$ of a graph $H$ is the minimum number of monochromatic labeled copies of $H$ in a 2-edge-coloring of the complete graph $K_n$. Define the Ramsey multiplicity constant $C(H)$ to be $$C(H):=\lim_{n \to +\infty}\frac{M(H,n)}{n^{v(H)}},$$ 
where $v(H)$ is the number of vertices of $H$. The limit exists: for \(n\ge v(H)\), the ratio
\(M(H,n)/(n)_{v(H)}\), where
\((n)_k=n(n-1)\cdots(n-k+1)\), is nondecreasing by averaging over induced complete subgraphs.
The random $2$-edge-coloring implies that $C(H)\leq 2^{1-e(H)}$ for every graph $H$, where $e(H)$ is the number of edges of $H$. If the Ramsey multiplicity constant of a graph $H$ is exactly $2^{1-e(H)}$, then we call $H$ a common graph, or simply $H$ is common. In other words, a graph \(H\) is common if the number of monochromatic
copies of \(H\) in a \(2\)-edge-coloring of a large complete graph is asymptotically minimized by the random coloring with an equal proportion of each color. For example, Goodman's formula~\cite{G59} shows that $K_3$ is common.

In 1962, Erd\H{o}s~\cite{E62} conjectured that every complete graph is common and later Burr and Rosta~\cite{BR80} extended this conjecture to every graph. Both conjectures failed to be true. Sidorenko~\cite{S89} proved that a triangle with a pendent edge is not common and Thomason~\cite{T89} proved that $K_4$ is not common. A general result from Jagger, {\v S}t'ov{\' i}{\v c}ek and Thomason~\cite{JST96} showed that every graph containing $K_4$ as a subgraph is not common. 

The classification of common graphs is a wide-open problem in extremal graph theory. For bipartite graphs, if $H$ satisfies Sidorenko's conjecture~\cite{S93}, then $H$ is common. 
We remark that an OpenAI preprint~\cite{OAI26} recently claimed a counterexample to Sidorenko's conjecture. Whether all bipartite graphs are common remains an open problem. 
For non-bipartite graphs, relatively few families of graphs are known to be common. 
Even for some graphs with few vertices, it is difficult to determine whether a graph is common. 
We direct
interested readers to~\cite{BMN22,BMN2024,GLLV,KKL26,KoL22,KNNVW22,KVW,LN23} for recent progress on this problem. 

The flag algebra method, introduced by Razborov~\cite{R07} in 2007, is a powerful computational framework for proving asymptotic inequalities between subgraph densities. It encodes small subgraph patterns into an algebraic structure and reduces the verification of such inequalities to semidefinite programming problems. This method has led to substantial progress in the study of common graphs with small size. For example, Hatami, Hladký, Král', Norine and Razborov~\cite{HHNRK12} used this method to prove that the \(5\)-wheel \(W_5\) is common. In~\cite{GLLV}, Grzesik, Lee, Lidick\'y and Volec used flag algebras to show that the $7$-wheel as well as all the connected $7$-vertex $K_4$-free non-$3$-colourable graphs are common. In their paper, they mentioned that the following two graphs $H_3$ and $H_4$ on the vertex set $\{1,2,3,4,5\}$ with $E(H_3)=\{12,13,23,14,45,25\}$ and $E(H_4)=\{12,14,24,34,15,25,35\}$ are common. The graph $H_3$ is the house graph, i.e., a pentagon with one chord and the graph $H_4$ is $K_4$ with one edge subdivided once (see Figure 1). Raghuvanshi~\cite{R16} proved commonality of $H_3$ using flag algebras. Grzesik, Lee, Lidicky and Volec~\cite{GLLV} give computer-assisted proofs for both graphs. While flag algebras are extremely powerful, such proofs are often long and do not by themselves explain why a small graph should be common. In~\cite{GLLV}, Grzesik, Lee, Lidick\'y and Volec said ``it would still be interesting to find simpler ‘human-friendly’
proofs of the commonality of $H_3$ or $H_4$”. In this short note, we answer their question by giving flag-algebra-free proofs that both graphs are common.
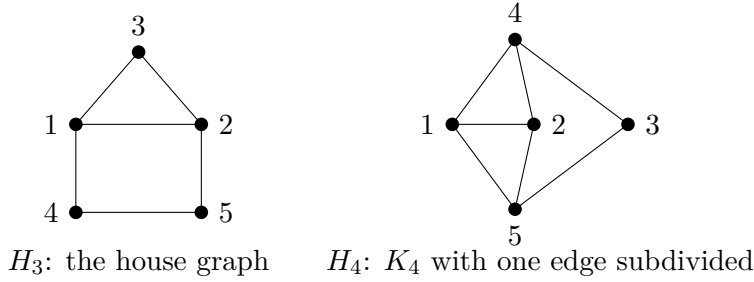
\begin{figure}[ht]
\centering
\begin{tikzpicture}[scale=.83,
  vertex/.style={circle,fill=black,inner sep=1.8pt}]
  \begin{scope}
    \coordinate (a) at (0,0);
    \coordinate (b) at (2,0);
    \coordinate (c) at (1,1.15);
    \coordinate (d) at (0,-1.4);
    \coordinate (e) at (2,-1.4);
    \draw (a)--(b)--(c)--(a)--(d)--(e)--(b);
    \node[vertex,label=left:$1$] at (a) {};
    \node[vertex,label=right:$2$] at (b) {};
    \node[vertex,label=above:$3$] at (c) {};
    \node[vertex,label=left:$4$] at (d) {};
    \node[vertex,label=right:$5$] at (e) {};
    \node at (1,-2.2) {$H_3$: the house graph};
  \end{scope}
  \begin{scope}[xshift=6cm]
    \coordinate (a) at (0,0);
    \coordinate (b) at (1.3,0);
    \coordinate (c) at (2.8,0);
    \coordinate (d) at (1,1.35);
    \coordinate (e) at (1,-1.35);
    \draw (a)--(b) (a)--(d)--(b)--(e)--(a)
          (d)--(c)--(e);
    \node[vertex,label=left:$1$] at (a) {};
    \node[vertex,label=right:$2$] at (b) {};
    \node[vertex,label=right:$3$] at (c) {};
    \node[vertex,label=above:$4$] at (d) {};
    \node[vertex,label=below:$5$] at (e) {};
    \node at (1.4,-2.2) {$H_4$: $K_4$ with one edge subdivided};
  \end{scope}
\end{tikzpicture}
\caption{The graphs $H_3$ and $H_4$.}
\end{figure}

\section{Preliminaries}
A homomorphism from a graph $H$ to $G$ is a map $f:V(H)\rightarrow V(G)$ such that $f(u)f(v)\in E(G)$ whenever $uv\in E(H)$. Let ${\rm Hom}(H,G)$ be the set of homomorphisms from $H$ to $G$. Define the homomorphism density $t(H,G)$ of $H$ in $G$ by $$t(H,G)=\frac{|{\rm Hom}(H,G)|}{v(G)^{v(H)}}.$$ 

To describe our proofs, it is convenient to work in the language of graph limits. For a comprehensive treatment of the theory of graph limits, we refer to~\cite{LL12}.

A kernel is a measurable function \(U:[0,1]^2\to \mathbb{R}\) such that $U(x,y)=U(y,x)$ for every $(x,y) \in [0,1]^2$. A graphon is a kernel $W:[0,1]^2 \to [0,1]$. 
For a finite graph $H:=(V(H),E(H))$ and a bounded kernel \(U\), the
homomorphism density of $H$ in $U$ is
\[
 t(H,U)=\int_{[0,1]^{v(H)}}
       \prod_{ij\in E(H)}U(x_i,x_j)\prod_{i\in V(H)}\d x_i.
\]
Integrals in this paper use Lebesgue measure on unit intervals. We use \(\norm{\cdot}\) to denote the \(L^2\)-norm: for a function
\(g:[0,1]\to\mathbb R\),
  $\norm{g}
  =\left(\int_0^1 |g(x)|^2\d x\right)^{1/2}$
and for a kernel \(K:[0,1]^2\to\mathbb R\), $\norm{K}
  =\left(\int_0^1\int_0^1 |K(x,y)|^2\d x\d y\right)^{1/2}.$

Fix a graphon $W$ and put
\[
\begin{aligned}
 U&=2W-1,& d(x)&=\int_0^1U(x,y)\d y,\\
 d_U&=\int_0^1 d(x)\d x,& f(x)&=\int_0^1U(x,y)d(y)\d y.
\end{aligned}
\]
Thus $U$ is a kernel with \(|U|\le1\) for every $(x,y)\in [0,1]^2$, \(|d(x)|\le1\) and \(|f(x)|\le1\) for $x\in [0,1]$. We also use the path
kernels
\[
\begin{aligned}
 P_2(x,y)&=\int_0^1U(x,z)U(z,y)\d z, \ \text{and}\\
 P_3(x,y)&=\int_{[0,1]^2}U(x,z)U(z,w)U(w,y)\d z\d w.
\end{aligned}
\]
A simple application of the Cauchy--Schwarz inequality yields that for
every real \(h\in L^2[0,1]\),
\begin{equation}\label{contraction}
 \int_0^1\left|\int_0^1U(x,y)h(y)\d y\right|^2\d x
 \le \left(\int U(x,y)^2\d x\d y\right)\norm{h}^2
 \le\norm{h}^2.
\end{equation}
In particular,
\begin{equation}\label{basic-bounds}
 d_U^2\le\norm{d}^2,\qquad
 \norm{f}^2\le\norm{d}^2,\qquad
 \norm{P_3}^2\le\norm{P_2}^2.
\end{equation}
The first inequality is Cauchy--Schwarz; the second follows by taking
\(h=d\) in \eqref{contraction}. For the third inequality in \eqref{basic-bounds}, fix \(y\in[0,1]\)
and set \(h_y(z)=P_2(z,y)\). Then
\(P_3(x,y)=\int_0^1 U(x,z)h_y(z)\d z\), so \eqref{contraction} gives
\[
  \int_0^1 |P_3(x,y)|^2\d x
  \le \norm{h_y}^2
  =\int_0^1 |P_2(z,y)|^2\d z.
\]
Integrating over \(y\in[0,1]\) gives
\(\norm{P_3}^2\le \norm{P_2}^2\).

Using the notation of homomorphism density, a graph is common if $$t(H,G)+t(H,\overline G)\geq 2^{1-e(H)}+o(1)$$ for every graph $G$, where $\overline G$ is the complement of $G$ and $o(1)\to 0$ as $v(G) \to +\infty$. By the standard theory of graph limits~\cite{LL12}, to prove a graph is common, it suffices to show
$$t(H,W)+t(H,1-W)\ge2^{1-e(H)}$$ for every graphon $W$. Finally, we introduce a useful even-subgraph expansion when studying common graphs (for example, see~\cite{KL22}). Setting $U=2W-1$, the standard multilinear expansion gives
\begin{equation}
\begin{aligned}
t(H,W)+t(H,1-W)&=\frac{t(H,1+U)+t(H,1-U)}{2^{e(H)}}=2^{1-e(H)}\left(1+\sum_{F\in \mathcal{E}^+(H)}t(F,U)\right),
\end{aligned}
\end{equation}
where $\mathcal{E}^{+}(H)$ is the set of all spanning subgraphs $F$ of $H$ with a positive even number of edges.

\section{The commonality of the graph \(H_3\)}

We prove the following theorem.
\begin{theorem}\label{thm:h3}
For every graphon \(W\),
\[
 t(H_3,W)+t(H_3,1-W)\ge\frac1{32}.
\]
Equality holds for \(W\equiv\tfrac12\). Thus \(H_3\) is common.
\end{theorem}

\begin{proof}
Expanding \(W=(1+U)/2\) and \(1-W=(1-U)/2\) gives
\begin{equation}\label{even-expansion}
 32\bigl[t(H_3,W)+t(H_3,1-W)\bigr]
 =1+\sum_{F\in \mathcal{E}^+(H_3)}
       t(F,U).
\end{equation}
There are nine adjacent edge pairs, each forming a two-edge path \(P_2\), since the degrees of \(H_3\) are \(3,3,2,2,2\). The remaining six edge pairs are disjoint. Their total contribution is \(9t(P_2,U)+6t(K_2\sqcup K_2,U)=9\norm{d}^2+6d_U^2,\) where we use the fact $t(K_2\sqcup K_2,U)=t(K_2,U)^2$.

The four-edge terms up to isomorphism are listed below. We still use the labels in Figure 1. In the third column, each
parenthesized pair specifies the two edges deleted from \(H_3\).
Isolated vertices are omitted from the graph descriptions.
\begin{center}
\small
\renewcommand{\arraystretch}{1.35}
\begin{tabular}{@{}l c >{\raggedright\arraybackslash}p{.38\textwidth} l@{}}
\toprule
Remaining graph & Count & Deleted edge pairs & Homomorphism density \\
\midrule
Four-edge path & \(7\) &
\((12,13)\), \((12,23)\), \((12,14)\), \((12,45)\), \((12,25)\),
\((13,25)\), \((23,14)\) & \(\norm{f}^2\) \\
Tree with degree sequence \(3,2,1,1,1\) & \(4\) &
\((13,14)\), \((13,45)\), \((23,45)\), \((23,25)\) & \(\int d(x)^2f(x)\d x\) \\
Four-cycle & \(1\) & \((13,23)\) & \(\norm{P_2}^2\) \\
Triangle with a pendent edge & \(2\) & \((14,45)\), \((45,25)\) &
\(\int U(x,y) P_2(x,y)d(x)\d x\d y\) \\
Disjoint union of a triangle and an edge & \(1\) & \((14,25)\) & \(d_U\int U(x,y)P_2(x,y)\d x\d y\) \\
\bottomrule
\end{tabular}
\end{center}

We explain two of the homomorphism densities in the fourth column. 
For the tree with degree sequence $3,2,1,1,1$, at the degree-three vertex, the two adjacent leaves give \(d(x)^2\) and the remaining two-edge branch gives \(f(x)\).
For a triangle with a pendent edge, integrate the leaf vertex of the pendent edge
and the opposite triangle vertex to obtain \(d(x)U(x,y)P_2(x,y)\).

For the six-edge term, keep vertices \(1,2\) at \(x,y\). The other
three vertices give \(P_2(x,y)P_3(x,y)\), so its homomorphism density is
\(\int U(x,y)P_2(x,y)P_3(x,y)\d x\d y\). 

The triangle terms and the six-edge term may have either sign. To estimate them together, introduce the auxiliary kernel
\[
 A(x,y)=P_3(x,y)+d(x)+d(y)+d_U.
\]

Consequently the expansion \eqref{even-expansion} and the table give the exact identity
\begin{equation}\label{h3-expansion}
\begin{aligned}
 32\bigl[t(H_3,W)+t(H_3,1-W)\bigr]-1&=9\norm{d}^2+6d_U^2+7\norm{f}^2+4\int d(x)^2f(x)\d x+\norm{P_2}^2\\
 &+\int U(x,y)P_2(x,y)A(x,y)\d x\d y.
\end{aligned}
\end{equation}
Because \(|U|\le1\), we have
\(2U(x,y)P_2(x,y)A(x,y)\ge-P_2(x,y)^2-A(x,y)^2\) for every $(x,y)\in[0,1]^2$. It follows that
\begin{equation}\label{h3-square-bound}
 \norm{P_2}^2+\int U(x,y)P_2(x,y)A(x,y)\d x\d y
 \ge\frac12\bigl(\norm{P_2}^2-\norm{A}^2\bigr).
\end{equation}

To expand $||A||_2^2$, note that 
 $\int_0^1P_2(x,y)\d y=f(x),$
 $\int P_3(x,y)\d x\d y=\int d(x)f(x)\d x$
and
\[
 \int d(x)P_3(x,y)\d x\d y=\int d(x)P_2(x,z)d(z)\d x\d z=\int_0^1
   \left(\int_0^1d(x)U(x,w)\d x\right)^2\d w
 =\norm{f}^2.
\]
The same identity holds with \(d(y)\) in place of \(d(x)\). Moreover, 
\[
 \int\bigl(d(x)+d(y)+d_U\bigr)^2\d x\d y
 =2\norm{d}^2+7d_U^2.
\]
Therefore
\begin{equation}\label{h3-A}
 \norm{A}^2
 =\norm{P_3}^2+4\norm{f}^2+2d_U\int d(x)f(x)\d x
       +2\norm{d}^2+7d_U^2.
\end{equation}
Substituting \eqref{h3-square-bound} and \eqref{h3-A} into
\eqref{h3-expansion} yields
\begin{equation}\label{h3-remainder}
\begin{aligned}
 &32\bigl[t(H_3,W)+t(H_3,1-W)\bigr]-1\\
 &\ge8\norm{d}^2+\frac52 d_U^2+5\norm{f}^2+4\int d(x)^2f(x)\d x-d_U\int d(x)f(x)\d x+\frac12\bigl(\norm{P_2}^2-\norm{P_3}^2\bigr).
\end{aligned}
\end{equation}
Here the last term is nonnegative by \eqref{basic-bounds}.
Since \(f(x)\ge-1\) and \(|d_U|\le1\),
\[
 \int d(x)^2f(x)\d x\ge-\norm{d}^2,\qquad
 \left|d_U\int d(x)f(x)\d x\right|
 \le\int|d(x)f(x)|\d x\le\frac12\bigl(\norm{d}^2+\norm{f}^2\bigr).
\]
We conclude from \eqref{h3-remainder} that
\[
\begin{aligned}
 32\bigl[t(H_3,W)+t(H_3,1-W)\bigr]-1
 \ge\frac72\norm{d}^2+\frac52 d_U^2+\frac92\norm{f}^2\ge0.
\end{aligned}
\]
For \(W\equiv\tfrac12\), both densities are \(2^{-6}\), proving
equality.
\end{proof}

\section{The commonality of the graph \(H_4\)}
We prove the following theorem.
\begin{theorem}\label{thm:h4}
For every graphon \(W\),
\[
 t(H_4,W)+t(H_4,1-W)\ge\frac1{64}.
\]
Equality holds for \(W\equiv\tfrac12\). Thus \(H_4\) is common.
\end{theorem}

\begin{proof}
Place vertices \(1,2,3\) at \(x,y,z\) and define
\[
\begin{aligned}
 A(x,y,z)&=P_2(x,y)+P_2(x,z)+P_2(y,z),\\
 B(x,y,z)&=\int U(x,w)U(y,w)U(z,w)\d w
                  +\frac32(d(x)+d(y))+d(z).
\end{aligned}
\]

To explain the expansion below, fix \(x,y,z\) and define the
common-neighbour factors
\[
\begin{aligned}
 C_W(x,y,z)&=\int W(x,w)W(y,w)W(z,w)\d w,\\
 C_{1-W}(x,y,z)&=\int (1-W(x,w))(1-W(y,w))(1-W(z,w))\d w.
\end{aligned}
\]
Since vertices \(4,5\) are precisely the common neighbours of
\(1,2,3\) in \(H_4\), we have
\[
\begin{aligned}
 t(H_4,W)=\int W(x,y)\,C_W(x,y,z)^2\d x\d y\d z,\quad  t(H_4,1-W)=\int (1-W(x,y))\,C_{1-W}(x,y,z)^2\d x\d y\d z.
\end{aligned}
\]
Now expand \(C_W\) using \(W=(1+U)/2\):
\[
\begin{aligned}
 C_W
 &=\frac18\int (1+U(x,w))(1+U(y,w))(1+U(z,w))\d w\\
 &=\frac18\Bigl(
 1+d(x)+d(y)+d(z)
 +P_2(x,y)+P_2(x,z)+P_2(y,z)+\int U(x,w)U(y,w)U(z,w)\d w
 \Bigr).
\end{aligned}
\]
With \(A(x,y,z),B(x,y,z)\) as defined above, this becomes
\[
 C_W=\frac18\left(1+A(x,y,z)+B(x,y,z)-\frac{d(x)+d(y)}2\right).
\]
Similarly, replacing \(U\) by \(-U\) gives
\[
 C_{1-W}=\frac18\left(1+A(x,y,z)-B(x,y,z)+\frac{d(x)+d(y)}2\right).
\]
Therefore,
\[
\begin{aligned}
 64\bigl[t(H_4,W)+t(H_4,1-W)\bigr]&=\frac12\int\Bigl[(1+U(x,y))
       \Bigl(1+A(x,y,z)+B(x,y,z)-\frac{d(x)+d(y)}2\Bigr)^2\\
       &+(1-U(x,y))
       \Bigl(1+A(x,y,z)-B(x,y,z)+\frac{d(x)+d(y)}2\Bigr)^2\Bigr]\d x\d y\d z,
\end{aligned}
\]
which gives
\begin{equation}\label{h4-expansion}
\begin{aligned}
 64\bigl[t(H_4,W)+t(H_4,1-W)\bigr]&=\int\Bigl[(1+A(x,y,z))^2+\Bigl(B(x,y,z)-\frac{d(x)+d(y)}2\Bigr)^2\\
 &+2U(x,y)(1+A(x,y,z))
       \Bigl(B(x,y,z)-\frac{d(x)+d(y)}2\Bigr)\Bigr]\d x\d y\d z.
\end{aligned}
\end{equation}

Expanding these squares and collecting terms explicitly yields
\begin{equation}\label{h4-polynomial}
\begin{aligned}
 &64\bigl[t(H_4,W)+t(H_4,1-W)\bigr]-1\\
&=\int\bigl[A(x,y,z)^2+B(x,y,z)^2+2U(x,y)A(x,y,z)B(x,y,z)\bigr]\d x\d y\d z\\
 &+\int\Bigl[2A(x,y,z)+2U(x,y)B(x,y,z)-B(x,y,z)(d(x)+d(y))\\
 &-U(x,y)(d(x)+d(y))-U(x,y)A(x,y,z)(d(x)+d(y))+\frac{(d(x)+d(y))^2}{4}\Bigr]\d x\d y\d z.
\end{aligned}
\end{equation}

We then deal with the second integral. Direct integrations yield
 $\int A(x,y,z)\d x\d y\d z=3\norm{d}^2,$ 
 $\int U(x,y)(d(x)+d(y))\d x\d y=2\norm{d}^2$ and
 $\int(d(x)+d(y))^2\d x\d y=2\norm{d}^2+2d_U^2.$
Integrating \(y,z\) first shows that
\[
 \int d(x)\left(\int U(x,w)U(y,w)U(z,w)\d w\right)\d x\d y\d z
 =\int d(x)^2f(x)\d x,
\]
and hence
\[
 \int B(x,y,z)(d(x)+d(y))\d x\d y\d z=2\int d(x)^2f(x)\d x+3\norm{d}^2+5d_U^2.
\]
We also have
\begin{equation}\label{h4-triangle}
\begin{aligned}
 &\int U(x,y)\left(\int U(x,w)U(y,w)U(z,w)\d w\right)
\d x\d y\d z=\int U(x,y)P_2(x,y)\frac{d(x)+d(y)}2\d x\d y,
\end{aligned}
\end{equation}
To justify this identity, integrating out \(z\) gives $\int U(x,y)U(x,w)U(y,w)d(w)\d x\d y\d w.$
The triangle factor is invariant under permutations of \(x,y,w\).
Relabeling \(w\) as \(x\), and then averaging the equivalent expressions
with \(d(x)\) and \(d(y)\), gives \eqref{h4-triangle}.
For the remaining terms, use \(\int P_2(x,z)\d z=f(x)\) and
\(\int U(x,y)d(y)\d y=f(x)\) to obtain
\[
\begin{aligned}
 &\int U(x,y)\Bigl(B(x,y,z)-\frac{A(x,y,z)(d(x)+d(y))}{2}\Bigr)\d x\d y\d z\\
 &=3\norm{d}^2+d_U^2
   -\frac12\int U(x,y)\bigl(P_2(x,z)+P_2(y,z)\bigr)(d(x)+d(y))\d x\d y\d z\\
 &=3\norm{d}^2+d_U^2-\int d(x)^2f(x)\d x-\norm{f}^2.
\end{aligned}
\]
Substituting these evaluations into \eqref{h4-polynomial}, the
remainder is
\[
\begin{aligned}
 &6\norm{d}^2
 +2\Bigl(3\norm{d}^2+d_U^2-\int d(x)^2f(x)\d x-\norm{f}^2\Bigr)\\
 &-\Bigl(2\int d(x)^2f(x)\d x+3\norm{d}^2+5d_U^2\Bigr)
       -2\norm{d}^2+\frac12\bigl(\norm{d}^2+d_U^2\bigr)\\
 &=\frac{15}2\norm{d}^2-\frac52 d_U^2-2\norm{f}^2-4\int d(x)^2f(x)\d x.
\end{aligned}
\]
Thus we have obtained the exact identity
\begin{equation}\label{h4-identity}
\begin{aligned}
 &64\bigl[t(H_4,W)+t(H_4,1-W)\bigr]-1\\
&=\int\bigl[A(x,y,z)^2+B(x,y,z)^2+2U(x,y)A(x,y,z)B(x,y,z)\bigr]\d x\d y\d z\\
&+\frac{15}2\norm{d}^2-\frac52 d_U^2-2\norm{f}^2-4\int d(x)^2f(x)\d x.
\end{aligned}
\end{equation}

We now estimate the right-hand side of \eqref{h4-identity}. Let
\[
Q:=\int_{[0,1]^3}
 \bigl(A(x,y,z)^2+B(x,y,z)^2
       +2U(x,y)A(x,y,z)B(x,y,z)\bigr)\d x\d y\d z.
\]
By the Cauchy--Schwarz inequality and \(|U|\le 1\),
\[
|f(x)|
=\left|\int_0^1 U(x,y)d(y)\d y\right|
\le
\left(\int_0^1 U(x,y)^2\d y\right)^{1/2}\norm{d}
\le \norm{d}
\]
for every $x\in[0,1]$.
Consequently,
\[
\int_0^1 d(x)^2f(x)\d x
\le \norm{d}\int_0^1 d(x)^2\d x
=\norm{d}^3.
\]
Together with \(d_U^2\le \norm{d}^2\) and \(\norm{f}^2\le \norm{d}^2\),
the exact identity \eqref{h4-identity} therefore yields
\begin{align*}
&64\bigl[t(H_4,W)+t(H_4,1-W)\bigr]-1=Q+\frac{15}{2}\norm{d}^2-\frac52 d_U^2
     -2\norm{f}^2-4\int_0^1 d(x)^2f(x)\d x\\
&\quad\ge Q+3\norm{d}^2-4\norm{d}^3.
\end{align*}

We next obtain a lower bound for \(Q\).
Since \(|U|\le 1\),
\begin{align*}
Q
&=\int_{[0,1]^3}
  \bigl((A(x,y,z)+U(x,y)B(x,y,z))^2
        +(1-U(x,y)^2)B(x,y,z)^2\bigr)
  \d x\d y\d z\\
&\ge \int_{[0,1]^3}(A(x,y,z)+U(x,y)B(x,y,z))^2\d x\d y\d z.
\end{align*}
Moreover, the definitions of $A(x,y,z)$ and $B(x,y,z)$ and \eqref{h4-triangle} give
\[
\int_{[0,1]^3}(A(x,y,z)+U(x,y)B(x,y,z))\d x\d y\d z
=
6\norm{d}^2+d_U^2+
\int_{[0,1]^2}U(x,y)P_2(x,y)d(x)\d x\d y.
\]
Using \(|U|,|P_2|\le 1\) and Cauchy--Schwarz, we have
\[
\left|
\int_{[0,1]^2}U(x,y)P_2(x,y)d(x)\d x\d y
\right|
\le \int_0^1 |d(x)|\d x
\le \norm{d}.
\]
Thus
\[
\int_{[0,1]^3}(A(x,y,z)+U(x,y)B(x,y,z))\d x\d y\d z
\ge 6\norm{d}^2-\norm{d}.
\]

Completing the square, and using \(\norm{d}\ge 0\), now gives
\begin{align*}
Q
&\ge \int_{[0,1]^3}(A(x,y,z)+U(x,y)B(x,y,z))^2\d x\d y\d z\\
&=\int_{[0,1]^3}
   \left(A(x,y,z)+U(x,y)B(x,y,z)-\frac{\norm{d}}{3}\right)^2
   \d x\d y\d z\\
&\qquad+\frac{2\norm{d}}{3}
  \int_{[0,1]^3}(A(x,y,z)+U(x,y)B(x,y,z))\d x\d y\d z
  -\frac{\norm{d}^2}{9}\\
&\ge \frac{2\norm{d}}{3}(6\norm{d}^2-\norm{d})-\frac{\norm{d}^2}{9}=4\norm{d}^3-\frac79\norm{d}^2.
\end{align*}
Combining the two estimates, we conclude that
\[
64\bigl[t(H_4,W)+t(H_4,1-W)\bigr]-1
\ge \frac{20}{9}\norm{d}^2
\ge 0.
\]
For \(W\equiv \tfrac12\), both densities equal \(2^{-7}\),
so equality holds in the desired commonality inequality.
\end{proof}

\section{Concluding Remarks}
The proofs also yield the stronger lower bounds $t(H_3,W)+t(H_3,1-W)\geq \frac{1}{32}+\frac{7}{64}\|d\|_2^2$ and $t(H_4,W)+t(H_4,1-W)\geq \frac{1}{64}+\frac{5}{144}\|d\|_2^2$. The arguments rely on explicit integral identities, the Cauchy--Schwarz inequality, and completion of squares, and can be verified without flag algebra machinery. Nevertheless, some of the positivity arguments are related to the positive semidefinite and sum-of-squares ideas underlying the flag algebra method. Our proofs express these ideas directly through elementary inequalities for graphon integrals.

The even-subgraph expansion method may become less effective for graphs with more edges, as the number of terms grows rapidly and their potentially negative contributions become harder to control. An interesting question, also raised by Joonkyung Lee in a talk, is whether the commonality of the $5$-wheel $W_5$ admits a proof without flag algebras. Our attempts to extend the present approach to this graph have so far been unsuccessful.

\section*{AI Declaration.}
The proofs in this paper were inspired by the author's work with ChatGPT-5.5 in August on counting specific two-colored subgraphs in large two-edge-colored complete graphs (the semi-inducibility problem). During this exploration, GPT-5.5 produced short flag-algebra-free proofs of upper bounds on the densities of specific two-colored subgraphs, using techniques such as the Cauchy--Schwarz inequality. During the preparation of this paper, the author used ChatGPT-6.1 to help check the expansions. The author's original proof for \(H_4\) involved case distinctions according to the value of \(\|d\|_2\). ChatGPT-6.1 suggested a more concise treatment of this step, which the author adopted after further discussion with the model. AI was also used to assist with language editing, correct typographical errors, and generate Figure 1. The author takes full responsibility for the correctness of the paper.


\begin{thebibliography}{99}
\small

\bibitem{BMN22}
N. Behague, N. Morrison, and J. A. Noel, \newblock{Common pairs of graphs}, \newblock{\emph{Combin. Probab. Comput.}} \textbf{34(5)} (2025), 649--670.

\bibitem{BMN2024}
N. Behague, N. Morrison, and J. A. Noel,
\emph{Off-diagonal commonality of graphs via entropy},
SIAM J. Discrete Math. \textbf{38} (2024), no.~3, 2335--2360.
%\bibitem{BCLSV12}
%C. Borgs, J. T. Chayes, L. Lov{\'a}sz, V. T. S{\'o}s, and K. Vesztergombi, \newblock{Convergent sequences of dense graphs II. Multiway cuts and statistical physics}, \newblock{\emph{Ann. Math.}} \textbf{176(1)} (2012), 151--219.

%\bibitem{BR65}
%G. R. Blakley and P. Roy, \newblock{A H{\"o}lder type inequality for symmetric matrices with nonnegative entries}, \newblock{\emph{Proc. Amer. Math. Soc.}} \textbf{16(6)} (1965), 1244--1245.

\bibitem{BR80}
S. A. Burr and V. Rosta, \newblock{On the Ramsey multiplicities of graphs---problems and recent results}, \newblock{\emph{J. Graph Theory}} \textbf{4(4)} (1980), 347--361.

%\bibitem{CFS10}
%D. Conlon, J. Fox, and B. Sudakov, \newblock{An approximate version of Sidorenko's conjecture}, \newblock{\emph{Geom. Funct. Anal.}} \textbf{20(6)} (2010), 1354--1366.

%\bibitem{CKLL18}
%D. Conlon, J. H. Kim, C. Lee, and J. Lee, \newblock{Some advances on Sidorenko's conjecture}, \newblock{\emph{J. London Math. Soc.}} \textbf{98(3)} (2018), 593--608.

%\bibitem{CL17}
%D. Conlon and J. Lee, \newblock{Finite reflection groups and graph norms}, \newblock{\emph{Adv. Math.}} \textbf{315} (2017), 130--165.

%\bibitem{CL21}
%D. Conlon and J. Lee, \newblock{Sidorenko's conjecture for blow-ups}, \newblock{\emph{Discrete Anal.}} (2021), Paper No. 2, 13 pp.

%\bibitem{CHL19}
%E. Cs{\'o}ka, T. Hubai, and L. Lov{\'a}sz, \newblock{Locally common graphs}, \newblock{\emph{J. Graph Theory}} \textbf{102(3)} (2023), 472--483.

%\bibitem{CY11}
%J. Cummings and M. Young, \newblock{Graphs containing triangles are not 3-common}, \newblock{\emph{J. Comb.}} \textbf{2(1)} (2011), 1--14.

\bibitem{E62}
P. Erd{\H{o}}s, \newblock{On the number of complete subgraphs contained in certain graphs}, \newblock{\emph{Magyar Tud. Akad. Mat. Kutat{\'o} Int. K{\"o}zl}} \textbf{7} (1962), 459--464.

%\bibitem{F08}
%J. Fox, \newblock{There exist graphs with super-exponential Ramsey multiplicity constant}, \newblock{\emph{J. Graph Theory}} \textbf{57(2)} (2008), 89--98.

%\bibitem{FW17}
%J. Fox and F. Wei, \newblock{On the local approach to Sidorenko's conjecture}, \newblock{\emph{Electron. Notes Discrete Math.}} \textbf{61} (2017), 459--465.

%\bibitem{HKKV22}
%R. Hancock, D. Kr{\'a}l', M. Krnc, and J. Volec, \newblock{Toward characterizing locally common graphs}, \newblock{\emph{Random Structures Algorithms}} \textbf{62(1)} (2023), 181--218.

%\bibitem{FW22}
%J. Fox and Y. Wigderson, \newblock{Ramsey multiplicity and the Tur{\'a}n coloring}, \newblock{\emph{Adv. Combin.}} (2023), Paper No. 2, 39 pp.

\bibitem{G59}
A. W. Goodman, \newblock{On sets of acquaintances and strangers at any party}, \newblock{\emph{Amer. Math. Monthly}} \textbf{66} (1959), 778--783.

\bibitem{GLLV}
A. Grzesik, J. Lee, B. Lidick{\'y}, and J. Volec, \newblock{On tripartite common graphs}, \newblock{\emph{Combin. Probab. Comput.}} \textbf{31(5)} (2022), 907--923.

%\bibitem{H}
%R. Halin, \newblock{S-functions for graphs}, \newblock{\emph{J. Geom.}} \textbf{8} (1976), 171--186.

\bibitem{HHNRK12}
H. Hatami, J. Hladk{\'y}, D. Kr{\'a}l', S. Norine, and A. A. Razborov, \newblock{Non-three-colourable common graphs exist}, \newblock{\emph{Combin. Probab. Comput.}} \textbf{21(5)} (2012), 734--742.

%\bibitem{H10}
%H. Hatami, \newblock{Graph norms and Sidorenko's conjecture}, \newblock{\emph{Israel J. Math.}} \textbf{175(1)} (2010), 125--150.

\bibitem{JST96}
C. Jagger, P. {\v S}{\v t}ov{\'\i}{\v c}ek, and A. Thomason, \newblock{Multiplicities of subgraphs}, \newblock{\emph{Combinatorica}} \textbf{16(1)} (1996), 123--141.

\bibitem{KKL26}
D. Kr{\'a}l', M. Krnc, and A. Lamaison, \newblock{Ramsey multiplicity of apices of trees}, \newblock{\emph{SIAM J. Discrete Math.}} \textbf{40(2)} (2026), 889--896.


\bibitem{KL22}
J. S. Kim and J. Lee, \newblock{Extended commonality of paths and cycles via Schur convexity}, \newblock{\emph{J. Combin. Theory Ser. B}} \textbf{166} (2024), 109--122.

\bibitem{KoL22}
S. Ko and J. Lee, \newblock{Common graphs with arbitrary connectivity and chromatic number}, \newblock{\emph{J. Combin. Theory Ser. B}} \textbf{162} (2023), 223--230.

\bibitem{KNNVW22}
D. Kr{\'a}l', J. A. Noel, S. Norin, J. Volec, and F. Wei, \newblock{Non-bipartite $k$-common graphs}, \newblock{\emph{Combinatorica}} \textbf{42(1)} (2022), 87--114.

\bibitem{KVW}
D. Kr{\'a}l', J. Volec, and F. Wei, \newblock{Common graphs with arbitrary chromatic number}, \newblock{\emph{Compos. Math.}} \textbf{161(3)} (2025), 594--634.

\bibitem{LN23}
J.-b. Lee and J. A. Noel, \newblock{Disconnected common graphs via supersaturation}, \newblock{\emph{Electron. J. Combin.}} \textbf{33(2)} (2026), Paper No. P2.1.

%\bibitem{JS21}
%J. Lee and B. Sch{\"u}lke, \newblock{Convex graphon parameters and graph norms}, \newblock{\emph{Israel J. Math.}} \textbf{242(2)} (2021), 549--563.

%\bibitem{LS11}
%X. Li and B. Szegedy, \newblock{On the logarithmic calculus and Sidorenko's conjecture}, \newblock{\emph{Acta Math. Hungar.}} \textbf{177} (2025), 450--463.

%\bibitem{LL11}
%L. Lov{\'a}sz, \newblock{Subgraph densities in signed graphons and the local Simonovits--Sidorenko conjecture}, \newblock{\emph{Electron. J. Combin.}} \textbf{18(1)} (2011), Paper No. P127, 21 pp.

\bibitem{LL12}
L. Lov{\'a}sz, \newblock{\emph{Large Networks and Graph Limits}}, Amer. Math. Soc. Colloq. Publ. \textbf{60}, American Mathematical Society, Providence, RI, 2012.

\bibitem{OAI26}
OpenAI, \newblock{A counterexample to Sidorenko's conjecture}, \newblock{\emph{OpenAI Math Release preprint}} (2026).

\bibitem{R16}
S. Raghuvanshi, \newblock{New results on Ramsey multiplicity and graph commonality}, Research Science Institute, MIT, 2016.

%\bibitem{RS84}
%Neil Robertson and Paul Seymour, \newblock{Graph minors \uppercase\expandafter{\romannumeral3}: Planar tree-width}, \newblock{\emph{J. Combin. Theory Ser. B}} \textbf{36} (1984), 49--64.

\bibitem{R07}
A. A. Razborov, \newblock{Flag algebras}, \newblock{\emph{J. Symbolic Logic}} \textbf{72(4)} (2007), 1239--1282.

\bibitem{S89}
A. Sidorenko, \newblock{Cycles in graphs and functional inequalities}, \newblock{\emph{Math. Notes}} \textbf{46(5)} (1989), 877--882.

\bibitem{S93}
A. Sidorenko, \newblock{A correlation inequality for bipartite graphs}, \newblock{\emph{Graphs Combin.}} \textbf{9(2-4)} (1993), 201--204.

%\bibitem{S931}
%A. Sidorenko, \newblock{Inequalities for functionals generated by bipartite graphs}, \newblock{\emph{Discrete Math. Appl.}} \textbf{2(5)} (1992), 489--504.

%\bibitem{S96}
%A. Sidorenko, \newblock{Randomness friendly graphs}, \newblock{\emph{Random Structures Algorithms}} \textbf{8(3)} (1996), 229--241.

%\bibitem{Sz15}
%B. Szegedy, \newblock{An information theoretic approach to Sidorenko's conjecture}, \newblock{\emph{arXiv:1406.6738}}.

\bibitem{T89}
A. Thomason, \newblock{A disproof of a conjecture of Erd{\H{o}}s in Ramsey theory}, \newblock{\emph{J. London Math. Soc.}} (2) \textbf{39(2)} (1989), 246--255.
\end{thebibliography}
\end{document}